\documentclass[reqno]{amsart}
\usepackage{amscd}
\usepackage{amssymb}
\usepackage[margin=1in,includeheadfoot]{geometry}
\usepackage{cite}
\usepackage [latin1]{inputenc}
\usepackage{tabularx,booktabs,tikz}
\usepackage{caption}
\usepackage{amsmath}
\usepackage{amsfonts}

\usepackage{amscd}
\usepackage{amsthm}
\usepackage{amssymb} \usepackage{latexsym}
\usepackage{eufrak}
\usepackage{euscript}
\usepackage{epsfig}
\usepackage{graphics}
\usepackage{array}
\usepackage{enumerate}
\usepackage{dsfont}
\usepackage{color}
\usepackage{wasysym}
\usepackage{hyperref}
\usepackage{pdfsync}

\def\beq{\begin{equation}}
\def\eeq{\end{equation}}
\def\ba{\begin{array}}
\def\ea{\end{array}}

\def\R{\mathbb R}

\newtheorem{thm}{Theorem}[section]
\newtheorem{lm}[thm]{Lemma}
\newtheorem{prop}[thm]{Proposition}

\theoremstyle{definition}
\newtheorem{rem}[thm]{Remark}

\theoremstyle{remark}

\begin{document}
\pagestyle{plain}
\title{Solutions to discrete nonlinear Kirchhoff-Choquard equations with power nonlinearity}

\author{Lidan Wang}
\email{wanglidan@ujs.edu.cn}
\address{Lidan Wang: School of Mathematical Sciences, Jiangsu University, Zhenjiang 212013, People's Republic of China}

\begin{abstract}
In this paper, we study the following Kirchhoff-Choquard equation
$$
-\left(a+b \int_{\mathbb{Z}^3}|\nabla u|^{2} d \mu\right) \Delta u+h(x) u=\left(R_{\alpha}\ast|u|^{p}\right)|u|^{p-2}u,\quad x\in \mathbb{Z}^3,
$$
where $a,\,b>0$, $\alpha \in(0,3)$ are constants and $R_{\alpha}$ is the Green's function of the discrete fractional Laplacian that behaves as the Riesz potential. Under some suitable assumptions on potential function $h$, for $p>2$, we first establish the existence of ground state
solutions  based on the Nehari manifold. Subsequently, for $p>4$, we obtain the existence of ground state sign-changing solutions
by adopting constrained minimization arguments on the sign-changing Nehari manifold.
\end{abstract}

\maketitle


\
\


\section{Introduction}

The Kirchhoff-type equation
\begin{equation}\label{z}
    -\left(a+b \int_{\mathbb{R}^3}|\nabla u|^{2} d \mu\right) \Delta u+h(x) u=g(x,u),\quad u\in H^1(\mathbb{R}^3),
\end{equation}
where $a,\,b>0$, has drawn lots of interest in recent years due to the appearance of $(\int_{\mathbb{R}^3}|\nabla u|^{2}\,d \mu)\Delta u$. See for examples \cite{CT,CT1,G,HZ,LY,LL} for the existence of ground state solutions to this equation, and \cite{FT,LL1,WT,WZ,YM,YT,ZW} for the existence of ground state sign-changing solutions.  Moreover, for the equations with logarithmic nonlinearity $g(x,u)=|u|^{p-2}u\log u^2$, we refer the readers to \cite{FW,GJ,HG,WTC}.

In many physical applications, the Choquard-type nonlinearity $g(x,u)=\left(I_\alpha \ast F(u)\right)f(u)$  appears naturally, where $I_\alpha$ is the Riesz potential. In particular, for $a=1$ and $b=0$, the equation (\ref{z}) turns into the well known Schr\"{o}dinger-Newton
equation, which has been studied extensively in the literature, see for examples \cite{GV,L0,L1,MV2,MV3}. Recently, for $\alpha\in (1, 3)$, Zhou and Zhu \cite{ZZ1} proved the existence of ground state solutions. Liang et al. \cite{LS} obtained the existence of multi-bump solutions. For $\alpha\in(0,3)$, Chen, Zhang and Tang \cite{CZ} proved the existence of ground state solutions under some hypotheses on $h$ and $f$. L\"{u} and Dai \cite{LD} established the existence and asymptotic behavior of ground state solutions by a Pohozaev-type constraint technique. Hu et al. \cite{H} obtained two classes of ground state solutions under the general Berestycki-Lions conditions on $f$. For $f(u)=|u|^{p-2}u$ with $p\in(2,3+\alpha)$, L\"{u} \cite{L} demonstrated the existence and asymptotic behavior 
of ground state solutions by the
Nehari manifold and the concentration compactness principle.
For more related works about the Choquard-type nonlinearity, we refer the readers to \cite{GT,HT,LP,LL,YG}.

Nowadays, many researchers  turn to study  differential equations on graphs, especially for the nonlinear elliptic equations. See for examples \cite{GLY,HSZ,HLW,HX,W2,ZZ} for the  discrete nonlinear  Schr\"{o}diner equations. For the  discrete nonlinear logarithmic equations, we refer the readers to \cite{CWY,CR,HJ,OZ,SY}. For the equations with Choquard-type nonlinearity, one can see \cite{LW,LZ1,LZ,W1}. Recently, L\"{u} \cite{L3} proved the existence of ground state solutions for a class of Kirchhoff equations on lattice graphs $\mathbb{Z}^3$. Later, Pan and Ji \cite{PJ} obtained the existence and convergence of ground state sign-changing solutions for a class of Kirchhoff equations on locally finite graphs. Very recently, Wang \cite{W4} proved the existence and convergence of ground state sign-changing solutions for nonlinear Kirchhoff equations with logarithmic nonlinearity on $\mathbb{Z}^3$. Moreover, Wang \cite{W3} established the existence of nontrivial solutions and ground
state solutions for Kirchhoff equations with general convolution on lattice graphs $\mathbb{Z}^3$. To the best of our knowledge, there is no existence results for the Kirchhoff-Choquard equations with power nonlinearity on graphs. Motivated by the works mentioned above, in this paper, we would like to study the Kirchhoff-type equations with power convolution nonlinearity on lattice graphs $\mathbb{Z}^3$ and discuss the existence of ground state solutions and ground state sign-changing solutions under different conditions on the potential function $h$.

Let us first give some notations. Let $C(\mathbb{Z}^{3})$ be the set of functions on $\mathbb{Z}^{3}$ and $C_{c}(\mathbb{Z}^{3})$ the set of functions on $\mathbb{Z}^{3}$ with finite support. We denote by $\ell^p(\mathbb{Z}^3)$ the space of $\ell^p$-summable functions on $\mathbb{Z}^3$. Moreover, for $u\in C(\mathbb{Z}^{3})$, we always write $
\int_{\mathbb{Z}^{3}} f(x)\,d \mu=\sum\limits_{x \in \mathbb{Z}^{3}} f(x),
$
where $\mu$ is the counting measure in $\mathbb{Z}^{3}$.

In this paper, we consider the following Kirchhoff-Choquard equation
\begin{equation}\label{aa}
-\left(a+b \int_{\mathbb{Z}^3}|\nabla u|^{2} d \mu\right) \Delta u+h(x) u=(R_{\alpha}\ast|u|^{p})|u|^{p-2}u,\quad x\in \mathbb{Z}^3, 
\end{equation}
where $a, b>0$ are constants, $\alpha\in(0,3)$ and $R_{\alpha}$ represents the Green's function of the discrete fractional Laplacian, see \cite{MC,W1}, $$ R_{\alpha}(x,y)=\frac{K_{\alpha}}{(2\pi)^3}\int_{\mathbb{T}^3}e^{i(x-y)\cdot k}\mu^{-\frac{\alpha}{2}}(k)\,dk,\quad y\in \mathbb{Z}^3,$$
which contains the fractional degree
$$K_\alpha=\frac{1}{(2\pi)^3}\int_{\mathbb{T}^3}\mu^{\frac{\alpha}{2}}(k)\,dk,\,\mu(k)=6-2\sum\limits_{j=1}^3 \cos(k_j),$$
where $\mathbb{T}^3=[0,2\pi]^3,\,k=(k_1,k_2,k_3)\in\mathbb{T}^3.$ Clearly, the Green's function $R_{\alpha}$ has no singularity at $x=y$. According to \cite{MC}, the Green's function $R_\alpha$  behaves as $|x-y|^{\alpha-3}$ for $|x-y|\gg 1$. 
Here $\Delta u(x)=\underset {y\sim x}{\sum}(u(y)-u(x))$ and $|\nabla u(x)|=\left(\frac{1}{2} \sum\limits_{y \sim x}(u(y)-u(x))^{2}\right)^{\frac{1}{2}}.$ 

Now we give assumptions on the potential function $h$:
\begin{itemize}
\item[($h_1$)] for any $x\in\mathbb{Z}^3$, there exists $h_0>0$ such that $h(x) \geq h_0$;
\item[($h_2$)] $h(x)$ is $\tau$-periodic in $x\in\mathbb{Z}^3$ with $\tau\in\mathbb{Z};$
\item[($h_3$)] there exists a point $x_0\in\mathbb{Z}^3$ such that $h(x)\rightarrow\infty$ as $|x-x_0|\rightarrow\infty.$
\end{itemize}

Let $H^{1}(\mathbb{Z}^3)$ be the completion of $C_c(\mathbb{Z}^3)$ with respect to the norm
\begin{eqnarray*}
\|u\|_{H^{1}}=\left(\int_{\mathbb{Z}^3}\left(|\nabla u|^2+u^2\right)\,d\mu\right)^{\frac{1}{2}}.
\end{eqnarray*}
Let $h(x) \geqslant h_{0}>0$, we introduce a new subspace
$$
H=\left\{u \in H^{1}(\mathbb{Z}^{3}): \int_{\mathbb{Z}^{3}} h(x) u^{2}\,d \mu<\infty\right\}
$$
with the norm
$
\|u\|=\left(\int_{\mathbb{Z}^{3}}\left(a|\nabla u|^{2}+h(x) u^{2}\right) d \mu\right)^{\frac{1}{2}},
$
where $a$ is a positive constant. The space $H$ is a Hilbert space with the inner product
$$
\langle u, v\rangle=\int_{\mathbb{Z}^3}\left(a\nabla u \nabla v+h(x) u v\right)\, d \mu.
$$
Since $h(x)\geq h_0>0$, we have $$\|u\|_2^2\leq \frac{1}{h_0}\int_{\mathbb{Z}^3} h(x)u^2(x)\,d\mu\leq \frac{1}{h_0}\|u\|^2.$$ Then for any $u \in H$ and $q\geq 2$, we have
\begin{equation}\label{ac}
\|u\|_{q} \leq \|u\|_{2}\leq C \|u\|.
\end{equation}

The energy functional $J(u): H\rightarrow\R$ associated to the equation (\ref{aa}) is given by
$$
J(u)=\frac{1}{2} \int_{\mathbb{Z}^{3}}\left(a|\nabla u|^{2}+h(x) u^{2}\right) d \mu+\frac{b}{4}\left(\int_{\mathbb{Z}^{3}}|\nabla u|^{2} d \mu\right)^{2}-\frac{1}{2p}\int_{\mathbb{Z}^{3}}(R_{\alpha}\ast|u|^p)|u|^p\,d\mu.
$$
By the discrete Hardy-Littlewood-Sobolev inequality (see Lemma \ref{lm1} below), one gets easily that $J(u)\in C^1(H,\mathbb{R})$ and, for any $\phi\in H$, 
$$
\left(J^{\prime}(u), \phi\right)=\int_{\mathbb{Z}^{3}}(a \nabla u \nabla \phi+h(x) u \phi)\,d \mu+b \int_{\mathbb{Z}^{3}}|\nabla u|^{2}\,d \mu \int_{\mathbb{Z}^{3}} \nabla u \nabla \phi\,d \mu-\int_{\mathbb{Z}^{3}}(R_{\alpha}\ast|u|^p)|u|^{p-2}u\phi \,d\mu.
$$
We say that $u\in H$ is a weak solution of the equation (\ref{aa}) if $u$ is a critical point of $J$, i.e. $J'(u)=0$. We define the Nehari manifold for the corresponding energy functional $J$
$$\mathcal{N}=\left\{u \in H \backslash\{0\}: (J^{\prime}(u),u)=0\right\},$$
and the sign-changing Nehari manifold
$$\mathcal{M}=\left\{u \in H: u^{ \pm} \neq 0 \text { and } (J^{\prime}(u), u^{+})=(J^{\prime}(u),u^{-})=0\right\},$$
where $u^{+}=\max \{u, 0\}$ and $u^{-}=\min \{u, 0\}.$
Moreover, we denote
$$c=\inf\limits_{u\in\mathcal{N}}J(u)\quad \text{and}\quad m=\inf\limits_{u\in \mathcal{M}}J(u).$$
We say that $u\in H$ is a ground state solution if
$J(u)=c$ and a ground state sign-changing solution if $J(u)=m.$

Now we state our main results.

\begin{thm}\label{t2}
 Let $p>2$. Assume that $(h_1)$ and $(h_2)$ hold. Then the equation (\ref{aa}) has a ground state solution.
    
\end{thm}

\begin{thm}\label{t0}
Let $p>4$. Assume that $(h_1)$ and $(h_3)$ hold. Then the equation (\ref{aa}) has a ground state sign-changing solution.
\end{thm}

The rest of this paper is organized as follows. In Section 2, we present some basic results on graphs. In Section 3, we prove the existence of ground state solutions based on the Nehari manifold and Lions lemma (Theorem \ref{t2}). In Section 4, we prove the existence of ground state sign-changing solutions by adopting constrained minimization arguments on the sign-changing Nehari manifold (Theorem \ref{t0}).

\section{Preliminaries} 
In this section, we introduce the basic settings on graphs and give some preliminary results that useful in our proofs.

Let $G=(V,E)$ be a connected, locally finite graph, where $V$ denotes the vertex set and $E$ denotes the edge set. We call vertices $x$ and $y$ neighbors, denoted by $x \sim y$, if there exists an edge connecting them, i.e. $(x, y) \in E$. For any $x,y\in V$, the distance $d(x,y)$ is defined as the minimum number of edges connecting $x$ and $y$, namely
$$d(x,y)=\inf\{k:x=x_0\sim\cdots\sim x_k=y\}.$$
Let $B_{r}(a)=\{x\in V: d(x,a)\leq r\}$ be the closed ball of radius $r$ centered at $a\in \mathbb{V}$. For brevity, we write $B_{r}:=B_{r}(0)$.

In this paper, we consider, the natural discrete model of the Euclidean space, the integer lattice graph.  The $3$-dimensional integer lattice graph, denoted by $\mathbb{Z}^3$, consists of the set of vertices $V=\mathbb{Z}^3$ and the set of edges $E=\{(x,y): x,\,y\in\mathbb{Z}^{3},\,\underset {{i=1}}{\overset{3}{\sum}}|x_{i}-y_{i}|=1\}.$
In the sequel, we denote $|x-y|:=d(x,y)$ on the lattice graph $\mathbb{Z}^{3}$.

For $u,v \in C(\mathbb{Z}^{3})$, we define the Laplacian of $u$ as
$$
\Delta u(x)=\sum\limits_{y \sim x}(u(y)-u(x)),
$$
 and the gradient form $\Gamma$ as
$$
\Gamma(u, v)(x)=\frac{1}{2} \sum\limits_{y \sim x}(u(y)-u(x))(v(y)-v(x)) .
$$
We write $\Gamma(u)=\Gamma(u, u)$ and denote the length of the gradient as
$$
|\nabla u|(x)=\sqrt{\Gamma(u)(x)}=\left(\frac{1}{2} \sum\limits_{y \sim x}(u(y)-u(x))^{2}\right)^{\frac{1}{2}}.
$$

The space $\ell^{p}(\mathbb{Z}^{3})$ is defined as
$
\ell^{p}(\mathbb{Z}^{3})=\left\{u \in C(\mathbb{Z}^{3}):\|u\|_{p}<\infty\right\},
$ where

$$
\|u\|_{p}= \begin{cases}\left(\sum\limits_{x \in \mathbb{Z}^{3}}|u(x)|^{p}\right)^{\frac{1}{p}}, &  1 \leq p<\infty, \\ \sup\limits_{x \in \mathbb{Z}^{3}}|u(x)|, & p=\infty.\end{cases}
$$

Moreover, we  define by $D^{1,2}(\mathbb{Z}^3)$ the completion of $C_c(\mathbb{Z}^3)$ under the norm $\|u\|_{D^{1,2}}=\int_{\mathbb{Z}^3}|\nabla u|^2\,d\mu.$
Since $|\nabla u(x)|^2=\frac{1}{2}\underset {y\sim x}{\sum}(u(y)-u(x))^{2}$, one gets easily that
\begin{equation}\label{0.2}
  \int_{\mathbb{Z}^3}|\nabla u|^2\,d\mu\leq C\| u \|^2_2.  
\end{equation}

Recall that, for a given functional $\Psi\in C^{1}(X,\mathbb{R})$, a sequence $\{u_k\}\subset X$ is a $(PS)_c$ sequence of the functional $\Psi$, if it satisfies, as $k\rightarrow\infty$,
\begin{eqnarray*}
\Psi(u_k)\rightarrow c, \qquad \text{in}~ X,\qquad\text{and}\qquad
\Psi'(u_k)\rightarrow 0, \qquad \text{in}~ X^{*}.
\end{eqnarray*}
where $X$ is a Banach space and $X^{*}$ is the dual space of $X$. 

The following discrete Hardy-Littlewood-Sobolev (HLS for abbreviation) inequality plays a key role in this paper, see \cite{LW,W1}.

\begin{lm}\label{lm1}
Let $0<\alpha <3,\,1<r,s<\infty$ and $\frac{1}{r}+\frac{1}{s}+\frac{3-\alpha}{3}=2$. We have the discrete
HLS inequality
\begin{equation}\label{bo}
\int_{V}(R_\alpha\ast u)(x)v(x)\,d\mu\leq C_{r,s,\alpha}\|u\|_r\|v\|_s,\quad u\in \ell^r(V),\,v\in \ell^s(V).
\end{equation}
And an equivalent form is
\begin{equation*}\label{p1}
\|R_\alpha\ast u\|_{\frac{3r}{3-\alpha r}}\leq C_{r,\alpha}\|u\|_r,\quad u\in \ell^r(V),
\end{equation*}
where $1<r<\frac{3}{\alpha}$.
\end{lm}

\begin{rem}\label{rm}

For $u\in H$ and $p\ge\frac{3+\alpha}{3}$, let $r=s=\frac{6}{3+\alpha}$ in Lemma \ref{lm1} and by (\ref{ac}), we obtain that
\begin{equation}\label{f1}
\int_{V}(R_{\alpha}\ast|u|^p)|u|^p\, \,d\mu\leq C \|u\|_{\frac{6p}{3+\alpha}}^{2p}\leq C\|u\|^{2p}.
\end{equation}

\end{rem}

We introduce the classical Br\'{e}zis-Lieb lemma \cite{BL}. Let $(\Omega, \Sigma, \tau)$ be a measure space, which consists of a set $\Omega$ equipped with a $\sigma$-algebra $\Sigma$ and a Borel measure $\tau:\Sigma\rightarrow[0,\infty]$. 

\begin{lm}\label{i}
Let $(\Omega, \Sigma, \tau)$ be a measure space and $\{u_k\}\subset L^{p}(\Omega, \Sigma, \tau)$ with $0<p<\infty$. If
$\{u_k\}$ is uniformly bounded in $L^{p}(\Omega)$ and $u_k\rightarrow u, \tau$-almost everywhere in $\Omega$,
then we have that
\begin{eqnarray*}
\underset{k\rightarrow\infty}{\lim}(\|u_k\|^{p}_{L^{p}(\Omega)}-\|u_k-u\|^{p}_{L^{p}(\Omega)})=\|u\|^{p}_{L^{p}(\Omega)}.
\end{eqnarray*}

\end{lm}

\begin{rem}\label{lm4}
If $\Omega$ is countable and $\tau$ is the counting measure $\mu$ in $\Omega$, then we get a discrete version of the Br\'{e}zis-Lieb lemma.

\end{rem}

We also have the following two discrete Br\'{e}zis-Lieb lemmas, see \cite{HLW,LW}.
\begin{lm}\label{ln}
Let $\{u_k\}$ be a bounded sequence in $D^{1,2}(V)$ and $u_k\rightarrow u$ pointwise in $V$. Then we have that
\begin{eqnarray*}
\underset{k\rightarrow\infty}{\lim}\left(\int_{V}|\nabla u_k|^2\,d\mu
-\int_{V}|\nabla (u_k-u)|^{2}\,d\mu\right)=\int_{V}|\nabla u|^{2}\,d\mu.
\end{eqnarray*}
\end{lm}

\begin{lm}\label{03}
Let $\{u_k\}$ be a bounded sequence in $\ell^{\frac{6p}{3+\alpha}}(V)$ with $p\geq 1$ and $u_k\rightarrow u$ pointwise in $V$. Then we have that
\begin{eqnarray*}
\underset{k\rightarrow\infty}{\lim}\left(\int_{V}(R_\alpha\ast|u_k|^{p})|u_k|^{p}\,d\mu
-\int_{V}(R_\alpha\ast|u_k-u|^{p})|u_k-u|^{p}\,d\mu\right)=\int_{V}(R_\alpha\ast|u|^{p})|u|^{p}\,d\mu.
\end{eqnarray*}
\end{lm}

We state two compactness results for $H$. We first present a discrete Lions lemma, which
denies a sequence $\left\{u_{k}\right\}\subset H$ to distribute itself over $V.$

\begin{lm}\label{lgh}
 Let $2\leq s<\infty$. Assume that $\left\{u_{k}\right\}$ is bounded in $H$ and

$$
\left\|u_{k}\right\|_{\infty} \rightarrow 0,\quad n \rightarrow\infty \text {. }
$$
Then, for any $s<t<\infty$,
$$
u_{k} \rightarrow 0,\quad  \text { in } \ell^{t}\left(V\right) \text {. }
$$   
\end{lm} 

\begin{proof}
 By (\ref{ac}), we get that $\{u_k\}$ is bounded in $\ell^{s}\left(V\right)$. Hence, for $s<t<\infty$, this result follows from an interpolation inequality
$$
\left\|u_{k}\right\|_{t}^{t} \leq\left\|u_{k}\right\|_{s}^{s}\left\|u_{k}\right\|_{\infty}^{t-s} .
$$
 
\end{proof} 

The following compactness result is well known, see \cite{ZZ}.
\begin{lm}\label{lgg}
Let $(h_1)$ and $(h_3)$ hold. Then for any $q\geq 2$, $H$ is continuously embedded into $\ell^{q}\left(V\right)$. That is, there exists a constant $C$ depending only on $q$ such that, for any $u \in H$,
$$
\|u\|_{q} \leq C\|u\|.
$$
Furthermore, for any bounded sequence $\left\{u_{k}\right\} \subset H$, there exists $u \in H$ such that, up to a subsequence, 
$$
\begin{cases}u_{k} \rightharpoonup u, & \text { in } H, \\ u_{k} \rightarrow u, & \text{pointwise~in}~ V, \\ u_{k} \rightarrow u, & \text { in } \ell^{q}\left(V\right) .\end{cases}
$$ 
\end{lm}

In order to study the sign-changing solutions, we need the following proposition, which can be seen in \cite{W4}. For completeness, we present the proof in the context.
\begin{prop}\label{o} Let $s,t>0$. Then for any $u\in H$, we have
   \begin{itemize}
      \item[(i)] $$
      \int_{V} \left|\nabla(su^{+}+tu^{-})\right|^2 d \mu=\int_{V} \left|\nabla (su^{+})\right|^2 d \mu+\int_{V}\left|\nabla (tu^{-})\right|^2 d \mu-stK_{V}(u),
      $$
 where $K_V(u)=\sum\limits_{x \in V} \sum\limits_{y \sim x}\left[u^{+}(x) u^{-}(y)+u^{-}(x) u^{+}(y)\right]\leq 0.$    
 \item[(ii)] $$
 \int_{V} \nabla\left(su^{+}+tu^{-}\right)\nabla (su^+)\, d \mu=\int_{V} \left|\nabla(s u^{+})\right|^2 d \mu-\frac{st}{2} K_{V}(u).
 $$
 \item[(iii)] $$
 \int_{V} \nabla\left(su^{+}+tu^{-}\right)\nabla (tu^-)\, d \mu=\int_{V} \left|\nabla (tu^{-})\right|^2 d \mu-\frac{st}{2} K_{V}(u).
$$
  \end{itemize} 
\end{prop}

\begin{proof}
(i) A direct calculation yields that
$$
\begin{aligned}
& \int_{V} \left|\nabla(su^{+}+tu^{-})\right|^2 d \mu \\
= & \frac{1}{2} \sum_{x \in V} \sum_{y \sim x}\left[\left(su^{+}+tu^{-}\right)(y)-\left(su^{+}+tu^{-}\right)(x)\right]^{2}\\
= & \frac{1}{2} \sum_{x \in V} \sum_{y \sim x}\left[\left(su^{+}(y)-su^{+}(x)\right)^{2}+\left(tu^{-}(y)-tu^{-}(x)\right)^{2}-2st\left[u^{+}(x) u^{-}(y)+u^{-}(x) u^{+}(y)\right]\right] \\
= & \int_{V} \left|\nabla (su^{+})\right|^2 d \mu+\int_{V}\left|\nabla (tu^{-})\right|^2 d \mu-stK_{V}(u).
\end{aligned}
$$

(ii) By a direct computation, we get that
$$
\begin{aligned}
&\int_{V} \nabla\left(su^{+}+tu^{-}\right)\nabla (su^+)\, d \mu \\= & \frac{1}{2} \sum_{x \in V} \sum_{y \sim x}\left[\left(su^{+}+tu^{-}\right)(y)-\left(su^{+}+tu^{-}\right)(x)\right]\left[su^{+}(y)-su^{+}(x)\right] \\
= & \frac{1}{2} \sum_{x \in V} \sum_{y \sim x}\left[\left(su^{+}(y)-su^+(x)\right)^{2}-st\left[u^{+}(x) u^{-}(y)+u^{-}(x) u^{+}(y)\right]\right] \\
= & \int_{V} \left|\nabla(s u^{+})\right|^2 d \mu-\frac{st}{2} K_{V}(u) .
\end{aligned}
$$

(iii) Similar to (ii), we obtain that
$$
\begin{aligned}
&\int_{V} \nabla\left(su^{+}+tu^{-}\right)\nabla (tu^-)\, d \mu \\= & \frac{1}{2} \sum_{x \in V} \sum_{y \sim x}\left[\left(su^{+}+tu^{-}\right)(y)-\left(su^{+}+tu^{-}\right)(x)\right]\left[tu^{-}(y)-tu^{-}(x)\right] \\
= & \frac{1}{2} \sum_{x \in V} \sum_{y \sim x}\left[\left(tu^{-}(y)-tu^-(x)\right)^{2}-st\left[u^{+}(x) u^{-}(y)+u^{-}(x) u^{+}(y)\right]\right] \\
= & \int_{V} \left|\nabla(tu^{-})\right|^2 d \mu-\frac{st}{2} K_{V}(u) .
\end{aligned}
$$

\end{proof}

The following lemma highlights a notable distinction between the discrete case and the continuous one.

\begin{lm}\label{2.5} Let $s,t>0$. Then for any $u \in H$, we have
\begin{itemize}
    \item[(i)] $$
\begin{aligned}
J(su^++tu^-)=&J\left(su^{+}\right)+J\left(tu^{-}\right)-\frac{a}{2}stK_{V}(u)+\frac{b}{4}s^2t^2K^2_V(u)+\frac{b}{2}\left\|\nabla (su^{+})\right\|_{2}^{2}\left\|\nabla (tu^{-})\right\|_{2}^{2}\\ &-\frac{b}{2}stK_V(u)\left(\left\|\nabla (su^{+})\right\|_{2}^{2}+\left\|\nabla (tu^{-})\right\|_{2}^{2}\right)-\frac{1}{p}\int_{V}(R_{\alpha}\ast|su^+|^p)|tu^-|^p\,d\mu.
\end{aligned}
$$
\item[(ii)] $$
\begin{aligned}
(J^{\prime}(su^++tu^-),su^{+}) =&(J^{\prime}\left(su^{+}\right),su^{+})-\frac{a}{2}stK_{V}(u)+\frac{b}{2}s^2t^2K^2_V(u)+b\left\|\nabla (su^{+})\right\|_{2}^{2}\left\|\nabla (tu^{-})\right\|_{2}^{2}\\&-\frac{b}{2}stK_V(u)\left(\left\|\nabla (su^{+})\right\|_{2}^{2}+\left\|\nabla (tu^{-})\right\|_{2}^{2}\right)
-bstK_V(u)\|\nabla (su^{+})\|^2_2\\&-\int_{V}(R_{\alpha}\ast|su^+|^p)|tu^-|^p\,d\mu.
\end{aligned}
$$

\item[(iii)] $$
\begin{aligned}
(J^{\prime}(su^++tu^-),tu^{-}) &=(J^{\prime}\left(tu^{-}\right),tu^{-})-\frac{a}{2}stK_{V}(u)+\frac{b}{2}s^2t^2K^2_V(u)+b\left\|\nabla (su^{+})\right\|_{2}^{2}\left\|\nabla (tu^{-})\right\|_{2}^{2}\\&-\frac{b}{2}stK_V(u)\left(\left\|\nabla (su^{+})\right\|_{2}^{2}+\left\|\nabla (tu^{-})\right\|_{2}^{2}\right)
-bstK_V(u)\|\nabla (tu^{-})\|^2_2\\&-\int_{V}(R_{\alpha}\ast|su^+|^p)|tu^-|^p\,d\mu.
\end{aligned}
$$
\end{itemize}

\end{lm}

\begin{proof}

(i)  Note that
$$
\begin{aligned}
&\int_{V}(R_{\alpha}\ast|su^++tu^-|^p)|su^++tu^-|^p\,d\mu\\=&\int_{V}(R_{\alpha}\ast|su^+|^p)|su^+|^p\,d\mu+2\int_{V}(R_{\alpha}\ast|su^+|^p)|tu^-|^p\,d\mu\\&+\int_{V}(R_{\alpha}\ast|tu^-|^p)|tu^-|^p\,d\mu.    
\end{aligned}
$$

It follows from Proposition \ref{o} that
$$
\begin{aligned}
&J(su^++tu^-)\\=&\frac{1}{2}
\int_{V} a|\nabla (su^++tu^-)|^2\,d\mu+\frac{1}{2}\int_{V}h(x)|su^++tu^-|^2\,d \mu+\frac{b}{4} \left(\int_{V}|\nabla(su^++tu^-)|^{2}\,d \mu \right)^2\\&-\frac{1}{2p}\int_{V}(R_{\alpha}\ast|su^++tu^-|^p)|su^++tu^-|^p\,d\mu\\=&\frac{a}{2}\left(\int_{V} \left|\nabla(s u^{+})\right|^2 d \mu+\int_{V} \left|\nabla(t u^{-})\right|^2 d \mu-stK_{V}(u)\right)+\frac{1}{2}\int_{V}h(x)\left(|su^+|^2+|tu^-|^2\right)\,d \mu\\&+ \frac{b}{4}\left(\int_{V} \left|\nabla (su^{+})\right|^2 d \mu+\int_{V}\left|\nabla (tu^{-})\right|^2 d \mu-stK_{V}(u)\right)^2\\&-\frac{1}{2p}\left(\int_{V}(R_{\alpha}\ast|su^+|^p)|su^+|^p\,d\mu+2\int_{V}(R_{\alpha}\ast|su^+|^p)|tu^-|^p\,d\mu+\int_{V}(R_{\alpha}\ast|tu^-|^p)|tu^-|^p\,d\mu.\right) \\=& \frac{1}{2}\int_{V}  a|\nabla(s u^{+})|^2 +h(x)|su^+|^2\,d \mu  +\frac{b}{4}\left(\int_{V} \left|\nabla (su^{+})\right|^2 d \mu\right)^2-\frac{1}{2p}\int_{V}(R_{\alpha}\ast|su^+|^p)|su^+|^p\,d\mu\\&+\frac{1}{2}\int_{V}  a|\nabla(tu^{-})|^2 +h(x)|tu^-|^2\,d \mu +\frac{b}{4}\left(\int_{V} \left|\nabla(tu^{-})\right|^2 d \mu\right)^2-\frac{1}{2p}\int_{V}(R_{\alpha}\ast|tu^-|^p)|tu^-|^p\,d\mu\\&-\frac{a}{2}stK_{V}(u)+\frac{b}{4}s^2t^2K^2_V(u)+\frac{b}{2}\left\|\nabla (su^{+})\right\|_{2}^{2}\left\|\nabla (tu^{-})\right\|_{2}^{2}\\&-\frac{b}{2}stK_V(u)\left(\left\|\nabla (su^{+})\right\|_{2}^{2}+\left\|\nabla (tu^{-})\right\|_{2}^{2}\right)
-\frac{1}{p}\int_{V}(R_{\alpha}\ast|su^+|^p)|tu^-|^p\,d\mu\\=&J\left(su^{+}\right)+J(tu^-)-\frac{a}{2}stK_{V}(u)+\frac{b}{4}s^2t^2K^2_V(u)+\frac{b}{2}\left\|\nabla (su^{+})\right\|_{2}^{2}\left\|\nabla (tu^{-})\right\|_{2}^{2}\\&-\frac{b}{2}stK_V(u)\left(\left\|\nabla (su^{+})\right\|_{2}^{2}+\left\|\nabla (tu^{-})\right\|_{2}^{2}\right)
-\frac{1}{p} \int_{V}(R_{\alpha}\ast|su^+|^p)|tu^-|^p\,d\mu.
\end{aligned}
$$

(ii) By Proposition \ref{o}, we get that
$$
\begin{aligned}
&(J^{\prime}(su^++tu^-),su^{+})\\=&
a\int_{V} \nabla (su^++tu^-)\nabla (su^{+})\,d\mu+\int_{V}h(x)(su^++tu^-)(su^{+})\,d \mu\\&+b \int_{V}|\nabla(su^++tu^-)|^{2}\,d \mu \int_{V} \nabla (su^++tu^-) \nabla (su^+)\,d \mu\\&-\int_{V}(R_{\alpha}\ast|su^++tu^-|^p)|su^++tu^-|^{p-2}(su^++tu^-)(su^+)\,d\mu\\=& a\left(\int_{V} \left|\nabla(s u^{+})\right|^2 d \mu-\frac{st}{2} K_{V}(u)\right)+\int_{V}h(x)(su^+)^2\,d \mu\\&+ b\left(\int_{V} \left|\nabla (su^{+})\right|^2 d \mu+\int_{V}\left|\nabla (tu^{-})\right|^2 d \mu-stK_{V}(u)\right)\left(\int_{V} \left|\nabla(s u^{+})\right|^2 d \mu-\frac{st}{2} K_{V}(u)\right)\\&-\int_{V}(R_{\alpha}\ast|su^+|^p)|su^+|^p\,d\mu-\int_{V}(R_{\alpha}\ast|su^+|^p)|tu^-|^p\,d\mu\\=& \int_{V}  a|\nabla(s u^{+})|^2 +h(x)(su^+)^2\,d \mu  +b\left(\int_{V} \left|\nabla (su^{+})\right|^2 d \mu\right)^2-\int_{V}(R_{\alpha}\ast|su^+|^p)|su^+|^p\,d\mu\\&-\frac{a}{2}stK_{V}(u)+\frac{b}{2}s^2t^2K^2_V(u)+b\left\|\nabla (su^{+})\right\|_{2}^{2}\left\|\nabla (tu^{-})\right\|_{2}^{2}\\&-\frac{b}{2}stK_V(u)\left(\left\|\nabla (su^{+})\right\|_{2}^{2}+\left\|\nabla (tu^{-})\right\|_{2}^{2}\right)
-bstK_V(u)\|\nabla (su^{+})\|^2_2\\&-\int_{V}(R_{\alpha}\ast|su^+|^p)|tu^-|^p\,d\mu\\=&(J^{\prime}\left(su^{+}\right),su^{+})-\frac{a}{2}stK_{V}(u)+\frac{b}{2}s^2t^2K^2_V(u)+b\left\|\nabla (su^{+})\right\|_{2}^{2}\left\|\nabla (tu^{-})\right\|_{2}^{2}\\&-\frac{b}{2}stK_V(u)\left(\left\|\nabla (su^{+})\right\|_{2}^{2}+\left\|\nabla (tu^{-})\right\|_{2}^{2}\right)
-bstK_V(u)\|\nabla (su^{+})\|^2_2\\&-\int_{V}(R_{\alpha}\ast|su^+|^p)|tu^-|^p\,d\mu.
\end{aligned}
$$
 (iii) By similar arguments as above, we obtain the desired result. We omit here.

\end{proof}

\begin{rem}
 Let $s=t=1$ in Corollary \ref{2.5}. Then for $u\in H\backslash\{0\}$, we have 
 \begin{itemize} 
     \item[(a)] the following results state the difference between the discrete and continuous cases.
     $$J_\lambda(u)\neq J_\lambda(u^+)+J_\lambda(u^-)+\frac{b}{2}\left\|\nabla u^{+}\right\|_{2}^{2}\left\|\nabla u^{-}\right\|_{2}^{2}-\frac{1}{p} \int_{V}(R_{\alpha}\ast|u^+|^p)|u^-|^p\,d\mu,$$
and $$(J^{\prime}(u),u^{\pm})\neq (J^{\prime}\left(u^{\pm}\right),u^{\pm})+b\left\|\nabla u^{+}\right\|_{2}^{2}\left\|\nabla u^{-}\right\|_{2}^{2}-\int_{V}(R_{\alpha}\ast|u^+|^p)|u^-|^p\,d\mu;$$

\item[(b)] in fact, from the proof of (ii) in Corollary \ref{2.5}, one sees that $(J'(u_k),u_k^{\pm})$ can also be expressed by
\begin{equation*}\label{3.8}
(J'(u_k),u_k^{\pm})=\|u_k^{\pm}\|^{2}+b\left\|\nabla u_k\right\|_{2}^{2}\left(\|\nabla u_k^{\pm}\|_{2}^{2}-\frac{1}{2}K_V(u_k)\right)-\frac{a}{2}K_{V}(u_k)-\int_{V}\left(R_{\alpha}*\left|u_k\right|^{p}\right)\left|u_k^{\pm}\right|^{p} \,d\mu.
\end{equation*}
 \end{itemize}
 
\end{rem}

\section{Proof of theorem \ref{t2}}

In this section, we prove the existence of ground state solutions to the equation (\ref{aa}). We first show that the functional $J(u)$ satisfies the mountain-pass geometry.
\begin{lm}\label{lm6}
Let $p>2$. Then
\begin{itemize}
\item[(i)] there exist $\theta, \rho>0$ such that $J(u)\geq\theta>0$ for $\|u\|=\rho$;
\item[(ii)] there exists $e\in H$ with $\|e\|>\rho$ such that $J(e)< 0$.
\end{itemize}
\end{lm}

\begin{proof}
(i) By the HLS inequality (\ref{f1}), we get that
\begin{eqnarray*}
J(u)&=&\frac{1}{2}\|u\|^{2}+\frac{b}{4}\left(\int_{V}|\nabla u|^{2} d \mu\right)^{2}-\frac{1}{2p}\int_{V}(R_{\alpha}\ast|u|^p)|u|^p\,d\mu\\&\geq&
\frac{1}{2}\|u\|^{2}-\frac{C}{2p}\|u\|^{2p}.
\end{eqnarray*}
Since $p>2$, there exist $\theta>0$ and $\rho>0$ small enough such that $J(u)\geq\theta>0$ for $\|u\|=\rho$.

\
\

(ii)
Note that $J(0)=0$. Moreover, for any $u\in H\backslash\{0\}$, as $t\rightarrow\infty$, one gets that
\begin{eqnarray*}
J(tu)=\frac{t^2}{2}\|u\|^{2}+\frac{bt^4}{4}\left(\int_{\mathbb{Z}^3}|\nabla u|^{2} d \mu\right)^{2}-\frac{t^{2p}}{2p}\int_{V}(R_{\alpha}\ast|u|^p)|u|^p\,d\mu\rightarrow
-\infty.
\end{eqnarray*}
Therefore, there exists $t_0>0$ large enough such that $\|e\|>\rho$ with $e=t_0 u$ and $J(e)<0$.

\end{proof}

\begin{lm}\label{lgj}
 Let $p>2$ and $\left(h_{1}\right)$ hold. Then for  $u \in \mathcal{N}$,
there exists $\eta>0$ such that $\|u\| \geq \eta$. 
\end{lm}

\begin{proof}
For $u \in\mathcal{N}$, by the HLS inequality (\ref{f1}), we have
$$
\begin{aligned}
0&=\left(J^{\prime}(u), u\right)\\ & =\|u\|^{2}+b\left(\int_{V}|\nabla u|^{2} d \mu\right)^{2}-\int_{V}(R_{\alpha}\ast|u|^p)|u|^p\,d\mu\\& \geq
\|u\|^{2}-C\|u\|^{2p}.
\end{aligned}
$$
Since $p>2$, we get easily that there exists a constant $\eta>0$ such that $\|u\| \geq \eta>0$.
\end{proof}

\
\

{\bf Proof of Theorem \ref{t2}. }   
By Lemma \ref{lm6}, one sees that $J$ satisfies the mountain-pass geometry. Hence there exists a sequence $\{u_k\}\subset H$ such that
$$J\left(u_k\right) \rightarrow c,\qquad\text{and}\qquad J^{\prime}\left(u_k\right) \rightarrow 0, \quad k\rightarrow\infty.$$  Note that $p>2$ and $b>0$. Then we get that
\begin{eqnarray}\label{ba}
\|u_k\|^2\nonumber&=&\frac{1}{p}\int_{V}(R_\alpha \ast |u_k|^p)|u_k|^p\,d\mu-\frac{b}{2}\left(\int_{V}\left|\nabla u_k\right|^{2} d \mu\right)^{2}+2c+o_k(1)\nonumber\\&\leq&\frac{1}{p} \left(\|u_k\|^2+b\left(\int_{V}\left|\nabla u_k\right|^{2} d \mu\right)^{2}+o_k(1)\|u_k\|\right)-\frac{b}{2}\left(\int_{V}\left|\nabla u_k\right|^{2} d \mu\right)^{2}+2c+o_k(1)\nonumber\\&=&\frac{1}{p}\|u_k\|^2+o_k(1)\|u_k\|+2c+o_k(1),
\end{eqnarray}
where $o_k(1)\rightarrow 0$ as $k\rightarrow\infty$. This inequality implies that $\{u_k\}$ is bounded in $H$. Hence there exists $u \in H$ such that $$u_k \rightharpoonup u,\quad\text{in~}H, \qquad\text{and}\qquad u_k\rightarrow u,\quad\text{pointwise~in}~V.$$
Moreover, we get that $\{u_k\}\subset\mathcal{N}$ since $0\leq |(J'(u_k),u_k)|\leq o_k(1)\|u_k\|\rightarrow 0.$

If \begin{equation}\label{hq}
 \left\|u_k\right\|_{\infty} \rightarrow 0,\quad k\rightarrow\infty,   
\end{equation} 
then by Lemma \ref{lgh}, we have that $u_k \rightarrow 0$ in $\ell^{t}\left(V\right)$ with $t>2$. Hence  \begin{eqnarray*}
\int_{V}(R_\alpha \ast |u_k|^p)|u_k|^p\,d\mu&\leq& C\|u_k\|^{2p}_{\frac{6p}{3+\alpha}}\\&\rightarrow& 0,\quad k\rightarrow\infty,
\end{eqnarray*}
that is
$$\int_{V}(R_\alpha \ast |u_k|^p)|u_k|^p\,d\mu=o_{k}(1).$$  Then
$$
\begin{aligned}
0=\left( J^{\prime}\left(u_k\right), u_k\right) & =\left\|u_k\right\|^{2}+b\left(\int_{V}\left|\nabla u_k\right|^{2} d \mu\right)^{2}-\int_{V}(R_\alpha \ast |u_k|^p)|u_k|^p\,d\mu\\
& \geq\left\|u_k\right\|^{2}+o_{k}(1),
\end{aligned}
$$
which implies that $\left\|u_k\right\| \rightarrow 0$ as $k\rightarrow\infty$. This contradicts $\left\|u_k\right\| \geq \eta>0$ in Lemma \ref{lgj}. Hence (\ref{hq}) does not hold, and hence there exists $\delta>0$ such that
\begin{equation}\label{hw}
\liminf _{k\rightarrow\infty}\left\|u_k\right\|_{\infty} \geq \delta>0,
\end{equation}
which implies that $u\neq 0$.
Therefore, there exists a sequence $\left\{y_k\right\} \subset V$ such that
\begin{equation*}
\left|u_k(y_k)\right| \geq \frac{\delta}{2}.
\end{equation*}
Let $\{z_k\}\subset V$ satisfy $\left\{y_k-z_k \tau\right\} \subset \Omega$, where $\Omega=[0, \tau)^{3}$. By translations, let $v_k(y):=u_k\left(y+z_k \tau\right)$. Then for any $v_k$,
\begin{equation*}
\left\|v_k\right\|_{l^{\infty}(\Omega)} \geq\left|v_k\left(y_k-z_k \tau\right)\right|=\left|u_k(y_k)\right| \geq \frac{\delta}{2}>0.
\end{equation*}
Since $h$ is $\tau$-periodic, $J$ and $\mathcal{N}$ are invariant under the translation, we obtain that $\left\{v_k\right\}$ is also a $(PS)_c$ sequence for $J$ and bounded in $H$. Then there exists $v\in H$ with $v\neq 0$ such that $$v_k \rightharpoonup v,\quad\text{in~}H, \qquad\text{and}\qquad v_k\rightarrow v,\quad \text{pointwise~in~}V.$$
Moreover, we also get that $\{v_k\}\subset\mathcal{N}.$ Now we prove that $v$ is a critical point of $J$. Let $A\geq 0$ be a constant such that $\int_{V}|\nabla v_k|^{2} d \mu \rightarrow A$ as $k\rightarrow\infty$. Note that
$$
\int_{V}|\nabla v|^{2} d \mu \leq \liminf _{k\rightarrow\infty} \int_{V}\left|\nabla v_k\right|^{2} d \mu=A.
$$
We claim that
\begin{equation}\label{55}
 \int_{V}|\nabla v|^{2} d \mu=A.   
\end{equation}
Arguing by contradiction, we assume that $\int_{V}|\nabla v|^{2} d \mu<A$. For any $\phi\in C_c(V),$ we have $( J^{\prime}\left(v_k\right),\phi)=o_k(1)$, namely
 \begin{equation}\label{hr}
\int_{V}\left(a \nabla v_k \nabla \varphi+h(x) v_k \varphi\right)\,d \mu+b \int_{V}\left|\nabla v_k\right|^{2} \,d\mu \int_{V} \nabla v_k \nabla \varphi\,d \mu-\int_{V}(R_\alpha \ast |v_k|^p)|v_k|^{p-2}v_k\phi\,d\mu=o_{k}(1).
\end{equation}
Let $k\rightarrow\infty$ in (\ref{hr}), we get that
\begin{equation}\label{ht}
\int_{V}(a \nabla v \nabla \varphi+h(x) v \varphi)\, d \mu+b A \int_{V} \nabla v \nabla \varphi\,d \mu-\int_{V}(R_\alpha \ast |v|^p)|v|^{p-2}v\phi\,d\mu=0.
\end{equation}
Since $C_c(V)$ is dense in $H$,  (\ref{ht}) holds for any $\phi\in H$. Let $\phi=v$ in (\ref{ht}), then we have 
$$
\begin{aligned}
( J'(v),v) & =\int_{V}\left(a|\nabla v|^{2}+h(x) v^{2}\right) d \mu+b\left(\int_{V}|\nabla v|^{2} d \mu\right)^{2}-\int_{V}(R_\alpha \ast |v|^p)|v|^{p-2}v\phi\,d\mu\\&<\int_{V}\left(a|\nabla v|^{2}+h(x) v^{2}\right) d \mu+b A\int_{V}|\nabla v|^{2} d \mu-\int_{V}(R_\alpha \ast |v|^p)|v|^{p-2}v\phi\,d\mu\\&=0.
\end{aligned}
$$
Let $$h(s)=\left(J^{\prime}(sv), s v\right),\quad s>0.$$ 
Then $h(1)=\left( J^{\prime}(v), v\right)<0$.

By the HLS inequality (\ref{f1}), we get, for $s>0$ small enough, that  
\begin{eqnarray}\label{he}
h(s)\nonumber&=&\left( J^{\prime}(sv), s v\right)\nonumber\\&=& s^2 \|v\|^2+s^4 b\left(\int_{V}|\nabla v|^{2} d \mu\right)^{2}-s^{2p}\int_{V}(R_\alpha \ast |v|^p)|v|^{p}\,d\mu\nonumber\\&\geq& s^2 \|v\|^2-Cs^{2p}\|v\|^{2p}\nonumber\\&>&0.
\end{eqnarray}
Hence, there exists $s_{0} \in(0,1)$ such that $h\left(s_{0}\right)=0$, i.e. $ \left( J^{\prime}\left(s_{0} v\right), s_{0} v\right)=0$. This means that $s_0v\in\mathcal{N}$, and hence $J\left(s_0 v\right)\geq c$. 
Then by Fatou's lemma, we obtain that
$$
\begin{aligned}
c & \leq J\left(s_0 v\right)=J\left(s_0 v\right)-\frac{1}{4}\left( J^{\prime}\left(s_0 v\right), s_0 v\right) \\
& =\frac{s_0^{2}}{4} \|v\|^2+(\frac{1}{4}-\frac{1}{2p})s_0^{2p}\int_{V}(R_\alpha \ast |v|^p)|v|^{p}\,d\mu\\
& <\frac{1}{4}\|v\|^2+(\frac{1}{4}-\frac{1}{2p})\int_{V}(R_\alpha \ast |v|^p)|v|^{p}\,d\mu\\
& \leq \liminf _{k\rightarrow\infty}\left[\frac{1}{4}\|v_k\|^2+(\frac{1}{4}-\frac{1}{2p})\int_{V}(R_\alpha \ast |v_k|^p)|v_k|^{p}\,d\mu\right] \\
& =\liminf _{k\rightarrow\infty}\left[J\left(v_k\right)-\frac{1}{4}\left( J^{\prime}\left(v_k\right), v_k\right)\right] \\
& =c.
\end{aligned}
$$
This is a contradiction. Hence, \begin{equation*}
\int_{V}\left|\nabla v_k\right|^{2} d \mu \rightarrow \int_{V}|\nabla v|^{2} d \mu=A.
\end{equation*}
The claim (\ref{55}) is completed. Then by (\ref{55}), (\ref{hr}) and (\ref{ht}), we get that $J^{\prime}(v)=0$, i.e. $v \in \mathcal{N}$. 
It remains to prove that $J(v)=c$. In fact, by Fatou's lemma, we obtain that
$$
\begin{aligned}
c & \leq J(v)-\frac{1}{4}\left( J^{\prime}(v), v\right) \\
& =\frac{1}{4}\|v\|^2+(\frac{1}{4}-\frac{1}{2p})\int_{V}(R_\alpha \ast |v|^p)|v|^{p}\,d\mu\\
& \leq \liminf _{k\rightarrow\infty}\left[\frac{1}{4}\|v_k\|^2+(\frac{1}{4}-\frac{1}{2p})\int_{V}(R_\alpha \ast |v_k|^p)|v_k|^{p}\,d\mu\right] \\
& =\liminf _{k\rightarrow\infty}\left[J\left(v_k\right)-\frac{1}{4}\left( J^{\prime}\left(v_k\right), v_k\right)\right] \\
& =c .
\end{aligned}
$$
Hence $J(v)=c$. \qed

\
\

\section{Proof of Theorem \ref{t0}}
In this section, we prove the existence of ground state sign-changing solutions to the equation (\ref{aa}). First, we prove a lemma that plays a crucial role in our proof of Theorem \ref{t0}.

\begin{lm}\label{4}
    
Let $p>4$. If $u \in H$ with $u^{ \pm} \neq 0$, then there exists a unique positive number pair $\left(s_{u}, t_{u}\right)$ such that $s_{u} u^{+}+t_{u} u^{-} \in \mathcal{M}$ and 
$$J(s_uu^++t_uu^-)=\max\limits_{s,t\geq 0} J(su^++tu^-).$$
\end{lm}

\begin{proof}
For any $u \in H$ with $u^{ \pm} \neq 0$, we define the function $f:[0, \infty) \times[0, \infty) \rightarrow \mathbb{R}$ given by
$$
f(s,t)=J\left(s u^{+}+tu^{-}\right).
$$
By a simple calculation, we get that
$$
\begin{aligned}
\nabla f(s,t)=&\left(f_s(s,t), f_t(s,t)\right) \\
 =&\left(\left(J^{\prime}\left(su^{+}+t u^{-}\right), u^{+}\right),\left( J^{\prime}\left(su^{+}+tu^{-}\right), u^{-}\right)\right).
\end{aligned}
$$
Note that $p>4$ and $K_V(u)<0$. Clearly,  for $(s,t)$ small enough, we have that
$$
\begin{aligned}
f(s,t)=& \frac{1}{2} s^{2}\left\|u^{+}\right\|^{2}+\frac{b}{4} s^{4}\|\nabla u^{+}\|_{2}^{4} -\frac{1}{2 p} s^{2 p} \int_{V}\left(R_{\alpha}* \left|u^{+}\right|^{p}\right) \left|u^{+}\right|^{p}\,d\mu\\&+\frac{1}{2} t^{2}\left\|u^{-}\right\|^{2}+\frac{b}{4} t^{4}\|\nabla u^{-}\|_{2}^{4} -\frac{1}{2 p} t^{2 p} \int_{V}\left(R_{\alpha}* \left|u^{-}\right|^{p}\right) \left|u^{-}\right|^{p}\,d\mu\\
&-\frac{a}{2}stK_{V}(u)+\frac{b}{4}s^2t^2K^2_V(u)+\frac{b}{2}s^2t^2\left\|\nabla u^{+}\right\|_{2}^{2}\left\|\nabla u^{-}\right\|_{2}^{2}\\&-\frac{b}{2}stK_V(u)\left(s^2\|\nabla u^{+}\|_{2}^{2}+t^2\|\nabla u^{-}\|_{2}^{2}\right)
-\frac{1}{p}s^pt^p\int_{V}(R_{\alpha}\ast|u^+|^p)|u^-|^p\,d\mu\\>&0.
\end{aligned}
$$
On the other hand, by elementary inequalities $2ab\leq (a^2+b^2)$ and $(a^p+b^p)\leq (a+b)^p\leq 2^{p-1}(a^p+b^p)$ with $a,b\geq 0$, we get that
$$
\begin{aligned}
f(s,t)=  & \frac{1}{2} s^{2}\left\|u^{+}\right\|^{2}+\frac{b}{4} s^{4}\|\nabla u^{+}\|_{2}^{4}-\frac{1}{2 p} s^{2 p} \int_{V}\left(R_{\alpha}* \left|u^{+}\right|^{p}\right) \left|u^{+}\right|^{p}\,d\mu\\&+\frac{1}{2} t^{2}\left\|u^{-}\right\|^{2}+\frac{b}{4} t^{4}\|\nabla u^{-}\|_{2}^{4} -\frac{1}{2 p} t^{2 p} \int_{V}\left(R_{\alpha}* \left|u^{-}\right|^{p}\right) \left|u^{-}\right|^{p}\,d\mu\\
&-\frac{a}{2}stK_{V}(u)+\frac{b}{4}s^2t^2K^2_V(u)+\frac{b}{2}s^2t^2\left\|\nabla u^{+}\right\|_{2}^{2}\left\|\nabla u^{-}\right\|_{2}^{2}\\&-\frac{b}{2}stK_V(u)\left(s^2\|\nabla u^{+}\|_{2}^{2}+t^2\|\nabla u^{-}\|_{2}^{2}\right)
-\frac{1}{p}s^pt^p\int_{V}(R_{\alpha}\ast|u^+|^p)|u^-|^p\,d\mu\\ \leq & \frac{1}{2} (s^{2}+t^2)\max\{\left\|u^{+}\right\|^{2},\left\|u^{-}\right\|^{2}\}+\frac{b}{4} (s^{2}+t^2)^2\max\left\{\|\nabla u^{+}\|_{2}^{4},\|\nabla u^{-}\|_{2}^{4}\right\}\\& -\frac{a}{2}(s^2+t^2)K_{V}(u)+\frac{b}{4}(s^2+t^2)^2 K^2_V(u)+\frac{b}{2}(s^2+t^2)^2\left\|\nabla u^{+}\right\|_{2}^{2}\left\|\nabla u^{-}\right\|_{2}^{2}\\&-\frac{b}{2}(s^2+t^2)^2K_V(u)\max\{\|\nabla u^{+}\|_{2}^{2},\|\nabla u^{-}\|_{2}^{2}\}\\&-\frac{1}{2^pp} (s^{2} +t^{2})^p\min\{\int_{V}\left(R_{\alpha}* \left|u^{+}\right|^{p}\right) \left|u^{+}\right|^{p}\,d\mu,\int_{V}\left(R_{\alpha}* \left|u^{-}\right|^{p}\right) \left|u^{-}\right|^{p}\,d\mu\} \\ \rightarrow &-\infty,\quad (s^2+t^2)\rightarrow\infty.
\end{aligned}
$$
Then by the continuity of $f$, $f(s,t)>0$ for $(s,t)$ small and $f(s,t)\rightarrow-\infty $ as $|(s,t)|\rightarrow\infty$, there exists a pair of  $(s_u,t_u)$ such that 
$$f(s_u,t_u)=\max\limits_{s,t\geq 0}f(s,t),$$
which is $$J(s_uu^++t_uu^-)=\max\limits_{s,t\geq 0} J(su^++tu^-).$$

We claim that $s_u,t_u>0.$ Indeed, without loss of generality, assuming the pair of $(s_u,0)$ is a maximum point of $f(s,t)$, we get that
\begin{equation}\label{3.2}
\begin{aligned}
f_{t}(s_u, t) =& t\left\|u^{-}\right\|^{2}+b t^{3}\|\nabla u^{-}\|_{2}^{4}-\frac{a}{2}s_uK_{V}(u)+\frac{b}{2}s_u^2tK^2_V(u)+b s_u^2 t\left\|\nabla u^{+}\right\|_{2}^{2}\left\|\nabla u^{-}\right\|_{2}^{2} \\&-\frac{b}{2}s_uK_V(u)\left(s_u^2\|\nabla u^{+}\|_{2}^{2}+t^2\|\nabla u^{-}\|_{2}^{2}\right)-bs_ut^2K_V(u)\|\nabla u^{-}\|_{2}^{2}\\
& -t^{2 p-1} \int_{V}\left(R_{\alpha}*\left|u^{-}\right|^{p}\right)\left|u^{-}\right|^{p} \,d\mu-s_u^{p} t^{p-1} \int_{V}\left(R_{\alpha}*\left|u^{+}\right|^{p}\right)\left|u^{-}\right|^{p} \,d\mu.
\end{aligned}
\end{equation}
Since $p>4$, we see that $f_t(s_u,t)>0$ for $t$ small enough, which implies that $f(s_u,t)$ is increasing for $t$ small. This contradicts that the pair of $(s_u,0)$ is a maximum point of $f(s,t)$. Hence $t_u>0.$ Similarly, we have that $s_u>0$. Consequently, $(s_u,t_u)$ is a positive maximum point of
$f(s,t).$ 

Next, we prove that $\left(s_{u} u^{+}+t_{u} u^{-}\right)\in \mathcal{M}$. Note that $\left(s_{u} u^{+}+t_{u} u^{-}\right)\in \mathcal{M}$ is equivalent to $\nabla f(s_u,t_u)=0$. Since the pair of $(s_u, t_u)$ is
a positive maximum point of $f(s,t)$, we get that 
$$f_s(s_u, t_u)=f_t(s_u, t_u)=0,$$
namely
$$\frac{1}{s_u}\left(J^{\prime}\left(s_uu^{+}+t_u u^{-}\right), s_uu^{+}\right)= \frac{1}{t_u}\left(J^{\prime}\left(s_uu^{+}+t_u u^{-}\right), t_uu^{-}\right)=0.$$
By the definition of sign-changing Nehari manifold, we obtain that $\left(s_{u} u^{+}+t_{u} u^{-}\right)\in \mathcal{M}$.

Finally, we prove the uniqueness of the pair $\left(s_{u}, t_{u}\right)$.  We claim that there is a unique positive constant $s_{u}$ such that $f_{s}\left(s_{u}, t_{u}\right)=$ 0. In fact, by contradiction, suppose there exist $s_{u_{1}}$ and $s_{u_{2}}$ with $s_{u_{1}}<s_{u_{2}}$ such that $f_{s}\left(s_{u_{1}}, t_{u}\right)=0$ and $ f_{s}\left(s_{u_{2}}, t_{u}\right)=0$. Therefore, we have
$$
\begin{aligned}
& 
s_{u_{1}}^{2 p-1} \int_{V}\left(R_{\alpha}*\left|u^{+}\right|^{p}\right)\left|u^{+}\right|^{p} \,d\mu+s_{u_{1}}^{p-1} t_u^{p} \int_{V}\left(R_{\alpha}*\left|u^{+}\right|^{p}\right)\left|u^{-}\right|^{p} \,d\mu \\ = & s_{u_{1}}\left\|u^{+}\right\|^{2}+b s_{u_{1}}^{3}\|\nabla u^{+}\|_{2}^{4}-\frac{a}{2}t_uK_{V}(u)+\frac{b}{2}s_{u_{1}}t_u^2K^2_V(u)+b s_{u_{1}} t_u^{2}\left\|\nabla u^{+}\right\|_{2}^{2}\left\|\nabla u^{-}\right\|_{2}^{2} \\&-\frac{b}{2}t_uK_V(u)\left(s_{u_{1}}^2\|\nabla u^{+}\|_{2}^{2}+t_u^2\|\nabla u^{-}\|_{2}^{2}\right)-bs_{u_{1}}^2t_uK_V(u)\|\nabla u^{+}\|_{2}^{2},
\end{aligned}
$$
and
$$
\begin{aligned}
& s_{u_{2}}^{2 p-1} \int_{V}\left(R_{\alpha}*\left|u^{+}\right|^{p}\right)\left|u^{+}\right|^{p} \,d\mu+s_{u_{2}}^{p-1} t_u^{p} \int_{V}\left(R_{\alpha}*\left|u^{+}\right|^{p}\right)\left|u^{-}\right|^{p} \,d\mu \\ = & s_{u_{2}}\left\|u^{+}\right\|^{2}+b s_{u_{2}}^{3}\|\nabla u^{+}\|_{2}^{4}-\frac{a}{2}t_uK_{V}(u)+\frac{b}{2}s_{u_{2}}t_u^2K^2_V(u)+b s_{u_{2}} t_u^{2}\left\|\nabla u^{+}\right\|_{2}^{2}\left\|\nabla u^{-}\right\|_{2}^{2} \\&-\frac{b}{2}t_uK_V(u)\left(s_{u_{2}}^2\|\nabla u^{+}\|_{2}^{2}+t_u^2\|\nabla u^{-}\|_{2}^{2}\right)-bs_{u_{2}}^2t_uK_V(u)\|\nabla u^{+}\|_{2}^{2}.
\end{aligned}
$$
Then the above two formulas imply that
\begin{align*}
& 0>\left(s_{u_{1}}^{2 p-4}-s_{u_{2}}^{2 p-4}\right) \int_{V}\left(R_{\alpha}*\left|u^{+}\right|^{p}\right)\left|u^{+}\right|^{p} \,d\mu+\left(s_{u_{1}}^{p-4}-s_{u_{2}}^{p-4}\right) t_{u}^{p} \int_{V}\left(R_{\alpha}*\left|u^{+}\right|^{p}\right)\left|u^{-}\right|^{p} \,d\mu\\=&
\left(\frac{1}{s_{u_{1}}^{2}}-\frac{1}{s_{u_{2}}^{2}}\right)\left\|u^{+}\right\|^{2}-\frac{a}{2}\left(\frac{1}{s_{u_{1}}^{3}}-\frac{1}{s_{u_{2}}^{3}}\right)t_uK_{V}(u)+\frac{b}{2}\left(\frac{1}{s_{u_{1}}^{2}}-\frac{1}{s_{u_{2}}^{2}}\right) t_{u}^{2}K^2_V(u)\\&+b\left(\frac{1}{s_{u_{1}}^{2}}-\frac{1}{s_{u_{2}}^{2}}\right)t_u^{2}\left\|\nabla u^{+}\right\|_{2}^{2}\left\|\nabla u^{-}\right\|_{2}^{2} -\frac{b}{2}t_uK_V(u)\left[\left(\frac{1}{s_{u_{1}}}-\frac{1}{s_{u_{2}}}\right)\|\nabla u^{+}\|_{2}^{2}+\left(\frac{1}{s_{u_{1}}^{3}}-\frac{1}{s_{u_{2}}^{3}}\right) t_{u}^{2}\|\nabla u^{-}\|_{2}^{2}\right]\\&-b\left(\frac{1}{s_{u_{1}}}-\frac{1}{s_{u_{2}}}\right) t_uK_V(u)\|\nabla u^{+}\|_{2}^{2}\\>&0, 
\end{align*}
where we have used the facts $s_{u_{1}}<s_{u_{2}}$, $p>4$ and $K_V(u)<0$. This is impossible. Hence there is a unique $s_{u}>0$ such that $f_{s}(s_u, t_u)=0.$ The claim is completed. By similar arguments as above, we obtain that there is a unique $t_{u}>0$ such that $f_{s}(s_u, t_u)=0.$ Hence there exists a unique positive constant pair $(s_u,t_u)$ such that $\left(s_{u} u^{+}+t_{u} u^{-}\right)\in \mathcal{M}$. 
\end{proof}

\begin{lm}\label{l9}
Let $p>4$, $(h_1)$ and $(h_3)$ hold. Then $m>0$ is achieved by some minimizer $u_0\in\mathcal{M}$.
\end{lm}

\begin{proof} 
Let $\{u_k\}\subset\mathcal{M}$ be a minimizing sequence such that
$$\lim\limits_{k\rightarrow\infty}J(u_k)=m.$$
Since $\{u_k\}\subset\mathcal{M}$, we get that $$(J'(u_k),u_k)=(J'(u_k),u_k^+)+(J'(u_k),u_k^-) =0.$$
A direct calculation yields that
$$
\begin{aligned}
\lim _{k \rightarrow\infty} J\left(u_{k}\right) & =\lim _{k \rightarrow\infty}\left[J\left(u_{k}\right)-\frac{1}{2p} (J^{\prime}\left(u_{k}\right), u_{k})\right] \\
& =\lim _{k \rightarrow\infty}\left[(\frac{1}{2}-\frac{1}{2p})\|u_k\|^{2}+(\frac{1}{4}-\frac{1}{2p})b\left\|\nabla u_k\right\|_{2}^{4}\right] \\
& = m.
\end{aligned}
$$
Then we have
$$(\frac{1}{2}-\frac{1}{2p})\|u_k\|^{2}\leq (\frac{1}{2}-\frac{1}{2p})\|u_k\|^{2}+(\frac{1}{4}-\frac{1}{2p})b\left\|\nabla u_k\right\|_{2}^{4}=m+o_k(1),$$
which implies that $\left\{u_{k}\right\}$ is bounded in $H$. 
By the HLS inequality (\ref{f1}), we have that
$$
\|u_k\|^{2}\leq\|u_k\|^{2}+b\left\|\nabla u_k\right\|_{2}^{4}=\int_{V}\left(R_{\alpha}*\left|u_k\right|^{p}\right)\left|u_k\right|^{p} \,d\mu\leq C\|u_k\|_{\frac{6p}{3+\alpha}}^{2p}\leq C_1\|u_k\|^{2p},
$$
which implies that $\|u_k\|\geq C>0$.
By Lemma \ref{lgg}, there exists $u\in H$ such that $u\neq 0$.
Moreover, by the boundedness of $\{u_k\}$ in $H$, Lemma \ref{ln}, Lemma \ref{lgg} and (\ref{0.2}), one gets that
$$\lim\limits_{k\rightarrow\infty}\int_{V}\left|\nabla u_{k}\right|^{2}\,d\mu=\int_{V}\left|\nabla u\right|^{2}\,d\mu,$$
and hence
\begin{equation}\label{8.3}
\lim\limits_{k\rightarrow\infty}\left(\int_{V}\left|\nabla u_{k}\right|^{2}\,d\mu \right)^2=\left(\int_{V}\left|\nabla u\right|^{2}\,d\mu \right)^2.
\end{equation}
 Then for $p>4$, we have that
 $$
\begin{aligned}
m=&\lim _{k \rightarrow\infty}\left[(\frac{1}{2}-\frac{1}{2p})\|u_k\|^{2}+(\frac{1}{4}-\frac{1}{2p})b\left\|\nabla u_k\right\|_{2}^{4}\right] \\ \geq 
& \lim _{k \rightarrow\infty}(\frac{1}{4}-\frac{1}{2p})b\left\|\nabla u_k\right\|_{2}^{4}\\=& (\frac{1}{4}-\frac{1}{2p})b\left\|\nabla u\right\|_{2}^{4}\\>& 0.
\end{aligned}
$$

Since $\{u_k\}$ is bonded in $H$, and hence $\{u_k^{\pm}\}$. Then by Lemma \ref{lgg}, there exists $u^{\pm}\in H$ such that 
\begin{equation}\label{kk}
\begin{cases}u_{k}^{\pm} \rightharpoonup u^{\pm}, & \text { weakly in } H, \\ u_{k}^{\pm} \rightarrow u^{\pm}, & \text { pointwise in } V, \\ u_{k}^{\pm} \rightarrow u^{\pm}, & \text { strongly in } \ell^{q}(V), ~q \in[2,\infty].\end{cases}
\end{equation}
 Moreover, note that
\begin{equation*}\label{3.8}
(J'(u_k),u_k^{\pm})=\|u_k^{\pm}\|^{2}+b\left\|\nabla u_k\right\|_{2}^{2}\left(\|\nabla u_k^{\pm}\|_{2}^{2}-\frac{1}{2}K_V(u_k)\right)-\frac{a}{2}K_{V}(u_k)-\int_{V}\left(R_{\alpha}*\left|u_k\right|^{p}\right)\left|u_k^{\pm}\right|^{p} \,d\mu=0.
\end{equation*}
Then by (\ref{ac}) and (\ref{f1}), we get that
$$
\begin{aligned}
C_0\|u_k^{\pm}\|_{\frac{6p}{3+\alpha}}^{2}\leq& \|u_k^{\pm}\|^{2}\\ < &\|u_k^{\pm}\|^{2}+b\left\|\nabla u_k\right\|_{2}^{2}\left(\|\nabla u_k^{\pm}\|_{2}^{2}-\frac{1}{2}K_V(u_k)\right)-\frac{a}{2}K_{V}(u_k)\\=&\int_{V}\left(R_{\alpha}*\left|u_k\right|^{p}\right)\left|u_k^{\pm}\right|^{p} \,d\mu\\ \leq &C\|u_k\|_{\frac{6p}{3+\alpha}}^{p}\|u_k^{\pm}\|_{\frac{6p}{3+\alpha}}^{p}\\\leq &C_1\|u_k^{\pm}\|_{\frac{6p}{3+\alpha}}^{p}.
\end{aligned}
$$
Since $u_{k}^{ \pm} \neq 0$, one has that
$$
\left\|u_{k}^{ \pm}\right\|_{\frac{6p}{3+\alpha}}\geq C>0,
$$
which implies that $u^{\pm}\neq 0.$

Since $\{u_k^{\pm}\}$ is bounded in $H$, by Lemma \ref{03}, (\ref{f1}) and (\ref{kk}), we get that
\begin{equation}\label{9.1}
\lim\limits_{k\rightarrow\infty}\int_{V}\left(R_{\alpha} *\left|u_{k}^{ \pm}\right|^{p}\right) \left|u_{k}^{ \pm}\right|^{p}\,d\mu=\int_{V}\left(R_{\alpha} *\left|u^{ \pm}\right|^{p}\right) \left|u^{ \pm}\right|^{p}\,d\mu.
\end{equation}
Now we prove
\begin{equation}\label{9.2}
\lim\limits_{k\rightarrow\infty}\int_{V}\left(R_{\alpha} * \left|u_{k}^{+}\right|^{p}\right)\left|u_{k}^{-}\right|^{p}\,d\mu =\int_{V}\left(R_{\alpha} *\left|u^{+}\right|^{p}\right) \left|u^{-}\right|^{p}\,d\mu.
\end{equation}
In fact, by an elementary inequality $|a^s-b^s|\leq C|a-b|(a^{s-1}+b^{s-1})$ with $a,b\geq 0, s\geq 1$, the HLS inequality (\ref{f1}), H\"{o}lder inequality and the boundedness of $\{u_k^{\pm}\}$ in $H$, we get that
$$\begin{aligned}
&\left|\int_{V}\left(R_{\alpha} * \left|u_{k}^{+}\right|^{p}\right)\left|u_{k}^{-}\right|^{p}\,d\mu -\int_{V}\left(R_{\alpha} *\left|u^{+}\right|^{p}\right) \left|u^{-}\right|^{p}\,d\mu\right|\\ \leq &
\int_{V}\left(R_{\alpha}* |u_k^+|^p\right) \left||u_{k}^-|^p-|u^-|^p\right|\,d\mu +\int_{V}\left(R_{\alpha} *\left|u^{-}\right|^{p}\right) \left||u_k^{+}|^p-|u^+|^p\right|\,d\mu\\\leq &C\int_{V}\left(R_{\alpha}* |u_k^+|^p\right)|u_k^--u^-| \left(|u_{k}^-|^{p-1}+|u^-|^{p-1}\right)\,d\mu \\&+C\int_{V}\left(R_{\alpha} *\left|u^{-}\right|^{p}\right)|u_k^+-u^+|\left(|u_k^{+}|^{p-1}+|u^+|^{p-1}\right)\,d\mu\\ \leq &
C\|u_k^+\|^p_{\frac{6p}{3+\alpha}}\left(\|u_k^-\|^{p-1}_{\frac{6p}{3+\alpha}}+\|u^-\|^{p-1}_{\frac{6p}{3+\alpha}}\right)\|u_k^--u^-\|_{\frac{6p}{3+\alpha}}\\&+C\|u^-\|^p_{\frac{6p}{3+\alpha}}\left(\|u_k^+\|^{p-1}_{\frac{6p}{3+\alpha}}+\|u^+\|^{p-1}_{\frac{6p}{3+\alpha}}\right)\|u_k^+-u^+\|_{\frac{6p}{3+\alpha}}\\\leq & C\|u_k^--u^-\|_{\frac{6p}{3+\alpha}}+C\|u_k^+-u^+\|_{\frac{6p}{3+\alpha}}\\ \rightarrow& ~0.
\end{aligned}
$$
Similar to (\ref{8.3}), we have that
\begin{equation}\label{8.0}
\lim\limits_{k\rightarrow\infty}\left(\int_{V}\left|\nabla u_{k}^{ \pm}\right|^{2}\,d\mu \right)^2=\left(\int_{V}\left|\nabla u^{ \pm}\right|^{2}\,d\mu \right)^2.
\end{equation}
Moreover, by (\ref{0.2}), (\ref{kk}) and the boundedness of $\{u_k^{\pm}\}$ in $H$, we have that
\begin{equation}\label{8.1}
\lim\limits_{k\rightarrow\infty}\int_{V}\left|\nabla u_{k}^{+}\right|^{2}\,d\mu \int_{V}\left|\nabla u_{k}^{-}\right|^{2}\,d\mu=\int_{V}\left|\nabla u^{+}\right|^{2}\,d\mu \int_{V}\left|\nabla u^{-}\right|^{2}\,d\mu.
\end{equation}
Since $u^{\pm}\neq 0$, by Lemma \ref{4}, there exists a unique pair $(s_{u}, t_{u})$ with $s_{u}, t_{u}>0$ such that $s_{u}u^{+}+t_{u}u^{-} \in \mathcal{M}$. By (\ref{kk})-(\ref{8.1}), Fatou lemma and Lemma \ref{4}, we obtain that
$$
\begin{aligned}
m& \leq J(s_{u} u^{+}+t_{u} u^{-}) \\
& \leq \liminf _{k\rightarrow\infty} J\left(s_{u} u_{k}^{+}+t_{u} u_{k}^{-}\right)  \\
& \leq \lim\limits_{k\rightarrow\infty} J\left(u_{k}\right) \\
& =m
\end{aligned}
$$
Therefore, we complete the proof by letting $u_0=s_{u} u^{+}+t_{u} u^{-}$.

\end{proof}

{\bf Proof of Theorem \ref{t0}.}  We only need to show that if $u_0\in\mathcal{M}$ with $u_0=s_uu^++t_uu^-$ satisfies $J(u_0)=m$, then $u_0$ is a sign-changing solution to the equation (\ref{aa}), that is, $J^{\prime}(u_0)=0$.

We assume by contradiction that $u_0 \in \mathcal{M}$ with $J(u_0)=m$, but $u$ is not a solution of the equation (\ref{aa}). Then we can find a function $\phi \in C_{c}(V)$ such that $\left(J'(u_0),\phi\right)= -2.$
This implies that, for some $\delta>0$ small enough,
\begin{equation}\label{7.0}
\left(J^{\prime}\left(s u^{+}+t u^{-}+\tau\phi\right), \phi\right) \leq-1,\quad \left|s-s_{u}\right|+\left|t-t_{u}\right|+|\tau| \leq \delta.
\end{equation}

Let $\eta: D\rightarrow[0,1]$ be a cut-off function satisfying
\[
\eta(s, t)=\left\{\begin{array}{l}
1, ~\text {if}~\left(s-s_{u}\right)^{2}+\left(t-t_{u}\right)^{2} \leq\left(\frac{\delta}{4}\right)^{2},  \\
0, ~\text {if}~\left(s-s_{u}\right)^{2}+\left(t-t_{u}\right)^{2} \geq\left(\frac{\delta}{2}\right)^{2},
\end{array}\right.
\]
where $D=\{(s,t):\left(s-s_{u}\right)^{2}+\left(t-t_{u}\right)^{2} \leq\delta^{2}\}.$

We consider $J'\left(su^++tu^-+\delta\eta(s,t)\phi\right)$. For $(s,t)\in D$, by (\ref{7.0}), we have
\begin{align}\label{7.1}
J(s u^{+}+t u^{-}+\delta \eta(s,t)\phi) & =J\left(s u^{+}+t u^{-}\right)+\int_{0}^{1}\left(J^{\prime}\left(s u^{+}+t u^{-}+\tau \delta\eta\left(s, t\right) \phi\right), \delta\eta\left(s, t\right) \phi\right) d \tau \nonumber\\
& \leq J\left(s u^{+}+t u^{-}\right)-\delta\eta\left(s, t\right).
\end{align}

For $(s, t) \in D,$ we denote $$v(s,t)=s u^{+}+t u^{-}+\delta \eta(s,t)\phi$$ and 
$$
\Phi(s, t)=\left(\left( J^{\prime}(v), v^{+}\right),\left(J^{\prime}(v), v^{-}\right)\right).
$$

If $\left(s-s_{u}\right)^{2}+\left(t-t_{u}\right)^{2}=\delta^{2}$, then $v(s,t)=s u^{+}+t u^{-}$. By Lemma \ref{4}, we have that $\Phi(s, t) \neq(0,0)$. As a result, the Brouwer topological degree $\operatorname{deg}(\Phi, \operatorname{int}(D), 0)$ is well-defined and $\operatorname{deg}(\Phi, \operatorname{int}(D), 0)=1$. Hence there exists a pair of positive number $\left(s_{0}, t_{0}\right)\in \operatorname{int}(D)$ such that $\Phi(s_0, t_0)=(0,0)$. This means that $v(s_0,t_0)\in \mathcal{M}$, and hence
\begin{equation}\label{7.3}
J(v(s_0,t_0))\geq m.
\end{equation}

If $\left(s_{0}, t_{0}\right) \neq\left(s_{u}, t_{u}\right)$, then by Lemma \ref{4}, $J\left(s_{0} u^{+}+t_{0} u^{-}\right)<J\left(s_{u} u^{+}+t_{u} u^{-}\right)=m$. Hence by (\ref{7.1}), we get that
$$
J\left(v\left(s_{0}, t_{0}\right)\right) \leq J\left(s_{0} u^{+}+t_{0} u^{-}\right)<m.
$$

If $\left(s_{0}, t_{0}\right)=\left(s_{u}, t_{u}\right)$, then $\eta(s_0,t_0)=1$. By (\ref{7.1}), we have
$$
J\left(v\left(s_{0}, t_{0}\right)\right) \leq J\left(s_{0} u^{+}+t_{0} u^{-}\right)-\delta \leq m-\delta<m.
$$
In any case, we have $J\left(v\left(s_{0}, t_{0}\right)\right)<m$, which contradicts (\ref{7.3}). The proof is completed.\qed

\
\

{\bf Declarations}

\
\

{\bf Conflict of interest:} The author declares that there are no conflicts of interests regarding the publication of
this paper.


\begin{thebibliography}{99}


\bibitem{BL} H. Brezis, E. Lieb, A relation between pointwise convergence of functions and convergence of functionals. Proc. Amer. Math.
Soc. 88 (1983), 486-490.

\bibitem{CWY} X. Chang, R. Wang and D. Yan, Ground states for logarithmic Schr\"{o}dinger equations on locally finite graphs. J. Geom. Anal. 33 (2023), no.7, Paper No. 211.

\bibitem{CR} X. Chang, V.D. R$\breve{a}$dulescu,  R. Wang and D. Yan, Convergence of least energy sign-changing solutions for logarithmic Schr\"{o}dinger equations on locally finite graphs. Commun. Nonlinear Sci. Numer. Simul. 125 (2023), Paper No. 107418.


\bibitem{CT} S. Chen, X. Tang, Berestycki-Lions conditions on ground state solutions for Kirchhoff-type problems with variable potentials. J. Math. Phys. 60 (2019), no. 12, 121509, 16 pp.


\bibitem{CT1} S. Chen, X. Tang, Ground state solutions for asymptotically periodic Kirchhoff-type equations with asymptotically cubic or super-cubic nonlinearities. Mediterr. J. Math. 14 (2017), no. 5, Paper No. 209, 19 pp.

\bibitem{CZ} S. Chen, B. Zhang and X. Tang, Existence and non-existence results for Kirchhoff-type problems with convolution nonlinearity.
Adv. Nonlinear Anal. 9 (2020), 148-167. 



\bibitem{FW} H. Fan, Y. Wang and L. Zhao, On a class of Kirchhoff type logarithmic Schr\"{o}dinger equations involving the critical or supercritical Sobolev exponent. J. Math. Phys. 65 (2024), no. 3, Paper No. 031509, 26 pp.




\bibitem{FT} R. Feng, C. Tang, Ground state sign-changing solutions for a Kirchhoff equation with asymptotically 3-linear nonlinearity. Qual. Theory Dyn. Syst. 20 (2021), no. 3, Paper No. 91, 19 pp.

\bibitem{GJ} Y. Gao, Y. Jiang, L. Liu and N. Wei, Multiple positive solutions for a logarithmic Kirchhoff type problem in $\mathbb{R}^3$. Appl. Math. Lett. 139 (2023), Paper No. 108539.

\bibitem{GV} M. Ghimenti, J. Van Schaftingen, Nodal solutions for the Choquard equation. J. Funct. Anal. 271 (2016), 107-135.


\bibitem {GLY} A. Grigor'yan, Y. Lin and Y. Yang,  Existence of positive solutions to
some nonlinear equations on locally finite graphs. Sci. China Math. 60 (2017), 1311-1324.

\bibitem{GT} G. Gu, X. Tang, The concentration behavior of ground states for a class of Kirchhoff-type problems with Hartree-type nonlinearity. Adv. Nonlinear Stud. 19 (2019), 779-795. 



\bibitem{G} Z. Guo,
Ground states for Kirchhoff equations without compact condition. J. Differential Equations 259 (2015), 2884-2902.

\bibitem{HSZ} X. Han, M. Shao and L. Zhao,  Existence and convergence of solutions for nonlinear biharmonic equations on graphs. J. Differential Equations 268 (2020), 3936-3961.



\bibitem{HJ} Z. He, C. Ji, Existence and multiplicity of solutions for the logarithmic Schr\"{o}dinger equation with a potential on lattice graphs. arXiv:2403.1586


\bibitem{HZ} X. He, W. Zou, Existence and concentration behavior of positive solutions for a Kirchhoff equation in $\mathbb{R}^3$. J. Differential Equations 252 (2012), 1813-1834.


\bibitem{HG} D. Hu, Q. Gao, Multiple solutions to logarithmic Kirchhoff equations, Acta Math. Sci. 42A (2) (2022), 401-417.

\bibitem{H} D. Hu, X. Tang, S. Yuan and Q. Zhang, Ground state solutions for Kirchhoff-type problems with convolution nonlinearity and Berestycki-Lions type conditions. Anal. Math. Phys. 12 (2022), no. 1, Paper No. 19, 27 pp.

\bibitem{HT} D. Hu, X. Tang and N. Zhang, Semiclassical ground state solutions for a class of Kirchhoff-Type problem with convolution nonlinearity. J. Geom. Anal. 32 (2022), no. 11, Paper No. 272, 37 pp.


\bibitem {HLW} B. Hua, R. Li and L. Wang, A class of semilinear elliptic equations on groups of polynomial growth. J. Differential Equations 363 (2023), 327-349.

\bibitem{HX} B. Hua, W. Xu, Existence of ground state solutions to some Nonlinear Schr\"{o}dinger equations
on lattice graphs. Calc. Var. Partial Differential Equations 62 (2023), no. 4, Paper No. 127, 17 pp.


\bibitem{LL1} B. Li, W. Long and A. Xia, Multiple positive and sign-changing solutions for a class of Kirchhoff equations. Commun. Contemp. Math. 26 (2024), no. 2, Paper No. 2250060, 26 pp.

\bibitem{LY} G. Li, H. Ye, Existence of positive ground state solutions for the nonlinear Kirchhoff type equations in $\mathbb{R}^3$. J. Differential Equations 257 (2014), 566-600.



\bibitem {LW} R. Li, L. Wang, The existence and convergence of solutions for the nonlinear Choquard equations on groups of polynomial growth. arXiv: 2208.00236

\bibitem{LS} S. Liang, M. Sun, S. Shi and S. Liang, On multi-bump solutions for the Choquard-Kirchhoff equations in $\mathbb{R}^N$.  Discrete Contin. Dyn. Syst. Ser. S 16 (2023), 3163-3193.

\bibitem{LP} S. Liang, P. Pucci and B. Zhang, Multiple solutions for critical Choquard-Kirchhoff type equations. Adv. Nonlinear Anal. 10 (2021), 400-419.

\bibitem{L0} E. Lieb, Existence and uniqueness of the minimizing solution of Choquard's nonlinear equation. Stud. Appl. Math. 57 (1977), 93-105.

\bibitem{L1} P. Lions, The Choquard equation and related questions. Nonlinear Anal. 4 (1980), 1063-1072.
\bibitem{LZY} G. Lin, Z. Zhou and J. Yu, Ground state solutions of discrete asymptotically linear Schr\"{o}dinger equations with bounded and non-periodic potentials. J. Dynamics and Differential Equations 32 (2020), 527-555.



\bibitem{LZ1} Y. Liu, M. Zhang, Existence of solutions for nonlinear biharmonic Choquard equations on weighted lattice graphs. J. Math. Anal. Appl. 534 (2024), no. 2, Paper No. 128079, 18 pp.


\bibitem{LZ} Y. Liu, M. Zhang, The ground state solutions to a class of biharmonic Choquard equations on weighted lattice graphs. Bull. Iranian Math. Soc. 50 (2024), no. 1, Paper No. 12, 17 pp.


\bibitem{L} D. L\"{u}, A note on Kirchhoff-type equations with Hartree-type nonlinearities. Nonlinear Anal. 99 (2014), 35-48.

\bibitem{LD} D. L\"{u}, S. Dai, Existence and asymptotic behavior of solutions for Kirchhoff equations with general Choquard-type nonlinearities. Z. Angew. Math. Phys. 74 (2023), no. 6, Paper No. 232, 15 pp.

\bibitem{LL} D. L\"{u}, Z. Lu, On the existence of least energy solutions to a Kirchhoff-type equation in $\mathbb{R}^3$. Appl. Math. Lett. 96 (2019), 179-186.

\bibitem{L3} W. L\"{u},
Ground states of a Kirchhoff equation with the potential on the lattice graphs.
Commun. Anal. Mech. 15 (2023), 792-810. 

\bibitem{MV2} V. Moroz, J. Van Schaftingen, Ground states of nonlinear Choquard equations: existence, qualitative properties and decay asymptotics. J. Funct. Anal. 265 (2013), 153-184.

\bibitem{MV3} V. Moroz, J. Van Schaftingen, Existence of groundstates for a class of nonlinear
Choquard equations. Trans. Amer. Math. Soc. 367 (2015), 6557-6579.



\bibitem{MC} T. Michelitsch, B. Collet, A. Riascos, A. Nowakowski and F. Nicolleau,
Recurrence of random walks with long-range steps generated by fractional Laplacian matrices
on regular networks and simple cubic lattices. J. Phys. A 50 (2017), no. 50, 505004, 29 pp.


\bibitem{PJ} G. Pan, C. Ji,
Existence and convergence of the least energy sign-changing solutions for nonlinear Kirchhoff equations on locally finite graphs.
Asymptot. Anal. 133 (2023), no. 4, 463-482.

\bibitem{OZ} X. Ou, X. Zhang,  Ground-state sign-changing homoclinic solutions for a discrete nonlinear $p$-Laplacian equation with logarithmic nonlinearity. Bound. Value Probl. (2024), Paper No. 6.


\bibitem{SY} M. Shao, Y. Yang and L. Zhao, 
Multiplicity and limit of solutions for logarithmic Schr\"{o}dinger equations on graphs.  
J. Math. Phys. 65 (2024), no. 4, Paper No. 041508.







\bibitem{W1} L. Wang, The ground state solutions to discrete nonlinear Choquard equations with Hardy weights.  Bull. Iranian Math. Soc. 49 (2023), no. 3, Paper No. 30, 29 pp.

\bibitem{W2} L. Wang, The ground state solutions of discrete nonlinear Schr\"{o}dinger equations with Hardy weights. Mediterr. J. Math. 21 (2024), no. 3, Paper No.78.


\bibitem{W3} L. Wang, Solutions to discrete nonlinear Kirchhoff-Choquard equations. Bull. Malays. Math. Sci. Soc. 47 (2024), no. 5, Paper No. 138.

\bibitem{W4} L. Wang, Sign-changing solutions to discrete nonlinear logarithmic Kirchhoff equations. arXiv:2407.09794


\bibitem{WZ} L. Wang, B. Zhang and K. Cheng, Ground state sign-changing solutions for the Schr\"{o}dinger-Kirchhoff equation in $\mathbb{R}^3$. J. Math. Anal. Appl. 466 (2018), no. 2, 1545-1569.

\bibitem{WTC} L. Wen, X. Tang and S. Chen, Ground state sign-changing solutions for Kirchhoff equations with logarithmic nonlinearity. Electron. J. Qual. Theory Differ. Equ. (2019), Paper No. 47.




\bibitem{WT} M. Wu, C. Tang, The existence and concentration of ground state sign-changing solutions for Kirchhoff-type equations with a steep potential well. Acta Math. Sci. Ser. B (Engl. Ed.) 43 (2023), 1781-1799.


\bibitem{YM} X. Yao, C. Mu, Existence of sign-changing solution with least energy for a class of Kirchhoff-type equations in $\mathbb{R}^N$. Electron. J. Qual. Theory Differ. Equ. (2017), Paper No. 32, 14 pp.

\bibitem{YT} C. Yang, C. Tang, Ground state sign-changing solutions for Schr\"{o}dinger-Kirchhoff equation with asymptotically cubic or supercubic nonlinearity. Qual. Theory Dyn. Syst. 22 (2023), no. 2, Paper No. 48, 25 pp.




\bibitem{YG} L. Yin, W. Gan and S. Jiang, Existence and concentration of ground state solutions for critical Kirchhoff-type equation involving Hartree-type nonlinearities. Z. Angew. Math. Phys. 73 (2022), no. 3, Paper No. 103, 19 pp.


\bibitem{ZW} Y. Zhao, X. Wu and C. Tang, Ground state sign-changing solutions for Schr\"{o}dinger-Kirchhoff-type problem with critical growth. J. Math. Phys. 63 (2022), no. 10, Paper No. 101503, 17 pp.


\bibitem{ZZ} N. Zhang, L. Zhao,   Convergence of ground state solutions for nonlinear Schr\"{o}dinger equations on graphs. Sci. China Math. 61 (2018), 1481-1494.


\bibitem{ZZ1} L. Zhou, C. Zhu, 
Ground state solution for a class of Kirchhoff-type equation with general convolution nonlinearity.
Z. Angew. Math. Phys. 73 (2022), no. 2, Paper No. 75, 13 pp.



\end{thebibliography}
\end{document}